%% file: Doubling_the_rate_arXiv.tex
\ifdefined\journalstyle
  \documentclass[graybox,envcountsame]{svmult}

\usepackage{makeidx}         
  \usepackage{graphicx}        
                         
  \usepackage{multicol}        
  \usepackage[bottom]{footmisc}%
  \usepackage{newtxtext}       %
  \usepackage[varvw]{newtxmath}

\else
  \documentclass[11pt,a4paper]{article}
  \usepackage{latexsym,amssymb,amsfonts,amsmath,epsfig,tabularx,stmaryrd,amsthm}
  \usepackage[usenames]{xcolor}
  \usepackage[font=scriptsize]{caption}

    \newtheorem{theorem}{Theorem}
  {\theoremstyle{definition}
  \newtheorem{definition}[theorem]{Definition}
}
  
  \newtheorem{proposition}[theorem]{Proposition}
  \newtheorem{corollary}[theorem]{Corollary}

  \usepackage{mathtools}
\fi

\usepackage{graphicx}
\usepackage{pdfsync}

\include{common-macros}

 {\theoremstyle{plain}} 
\usepackage{pdfsync}
\definecolor{darkred}{RGB}{139,0,0}
\definecolor{darkgreen}{RGB}{0,100,0}
\definecolor{darkmagenta}{RGB}{139,0,139}
\definecolor{darkorange}{RGB}{220,110,20}
\definecolor{gray}{rgb}{0.66, 0.66, 0.66}

\ifdefined\journalstyle
  \title*{Doubling the rate -- improved error bounds for orthogonal projection with application to interpolation}
  \titlerunning{Doubling the rate -- improved error bounds for orthogonal projection}
  \author{Ian H.~Sloan and Vesa Kaarnioja}
 \institute{
  Ian H.~Sloan 
    \at School of Mathematics and Statistics, UNSW Sydney, Sydney NSW 2052, Australia \\
    \email{
    i.sloan@unsw.edu.au}\\ 
     Vesa  Kaarnioja \at Department of Mathematics and Computer Science, Free University of Berlin, 14195 Berlin, Germany\\
 \email{vesa.kaarnioja@fu-berlin.de}
} 

\else

 \title{Doubling the rate -- improved error bounds for orthogonal projection with application to interpolation}

 \author{Ian H. Sloan\footnote{School of Mathematics and Statistics, UNSW Sydney, Sydney NSW 2052,
 Australia. Email: {\tt i.sloan@unsw.edu.au}}
  \and Vesa Kaarnioja\footnote{Department of Mathematics and Computer Science, Free
  University of Berlin, Arnimallee 6, 14195 Berlin, Germany. Email: {\tt vesa.kaarnioja@fu-berlin.de}}
 }
 \date{\today}

\fi

\begin{document}

\maketitle

\begin{abstract}Convergence rates for $L_2$ approximation in a Hilbert space $H$  are a central theme in numerical analysis. The present work is inspired by Schaback (\emph{Math.~Comp.}, 1999), who showed, in the context of best pointwise approximation for radial basis function interpolation, that the convergence rate for sufficiently smooth functions can be doubled, compared to the best rate for functions in the ``native space'' $H$. Motivated by this, we obtain a general result for $H$-orthogonal projection onto a finite dimensional subspace of $H$: namely, that any known $L_2$ convergence rate for all functions in $H$ translates into a doubled $L_2$ convergence rate for functions in a smoother normed space $B$, along with a similarly improved error bound in the $H$-norm, provided that $L_2$, $H$ and $B$ are suitably related.  As a special case we improve the known $L_2$ and $H$-norm  convergence rates for kernel interpolation in reproducing kernel Hilbert spaces, with particular attention to a recent study (Kaarnioja, Kazashi, Kuo, Nobile, Sloan, \emph{Numer.~Math.}, 2022) of  periodic kernel-based interpolation at lattice points applied to  parametric partial differential equations.  A second application is to radial basis function interpolation for general  conditionally positive definite basis functions, where again the $L_2$ convergence rate is doubled, and the convergence rate in the native space norm is similarly improved, for all functions in a smoother normed space $B$.
\end{abstract}

\section{Introduction}\label{sec:1}

A common situation in computational mathematics and approximation theory is that one seeks to approximate a real-valued function $f$ defined on a domain $D$, and belonging to some Hilbert space $H$  continuously embedded in $L_2(D)$, by the  $H$-orthogonal projection $Pf$ of $f$ onto a finite-dimensional subspace $V\subset H$. Many error bounds holding for all functions $f \in H$ are known.  Our aim in this paper is to provide improved error bounds for functions $g$ lying in a ``smoother'' normed space $B$ continuously embedded in $H$: we prove, under an appropriate relation between the spaces $L_2$, $H$ and $B$ (see \eqref{eq:norms_relation} below),  a doubled rate of convergence of $\|g - Pg\|_{L_2}$ for $g\in B$, and an analogous improved rate of convergence for the error in the  $H$ norm,  $\|g - Pg\|_H$.

The finite-dimensional spaces we have in mind are varied: for example, $V$ might be a linear space of polynomials or splines; or radial basis functions of given shape and a fixed set of nodes on $\mathbb{R}^d$; or a space of wavelets or kernels.  The domain might be $\mathbb{R}^d$ or an  open subset of $\mathbb{R}^d$, or a manifold such as $S^{d-1}:= \{ \bsx \in \mathbb{R}^d : \|\bsx\|_{\ell_2} =1\}$.

For very many Hilbert spaces $H$ and subspaces $V$ error bounds  of the following form are known for $\|f - P f\|_{L_2}$:
\begin{equation}\label{eq:first_ass}
\|f -P f\|_{L_2}  \le c n^{-\kappa} \|f\|_H \; \mbox{ for all } f \in H,
\end{equation}
where $n := \dim(V)$, $\|\cdot \|_H$ denotes the norm in $H$, and $\kappa>0$ is a known constant.   Results of this kind lie at the heart of computational mathematics.  If $\kappa$ is the supremum of all such constants then the result is one of best approximation.

The main aim of this paper is to show that if $H$ and $V$ are such that \eqref{eq:first_ass} holds, and if $B$ is an appropriate subspace of $H$ (see \eqref{eq:norms_relation} below), then an $L_2$ error bound with a doubled convergence rate holds for all functions in $B$:  specifically, we will prove
\begin{equation*}
\|g-P g\|_{L_2} \le c^2 n^{-2\kappa} \|g\|_B \; \mbox{ for all } g \in B.
\end{equation*}
Moreover, under the same condition the error bound in the $H$ norm will be shown to have the same bound as in \eqref{eq:first_ass}, specifically 
\begin{equation*}
\|g-P g\|_H \le c n^{-\kappa} \|g\|_B \; \mbox{ for all } g \in B.
\end{equation*}

The relation between the linear spaces $L_2$, $H$ and $B$ needed in this work, in addition to the embedding of $H$ in $L_2$ and of $B$ in $H$,  is 
 \begin{equation}\label{eq:norms_relation}
 |\langle f,g\rangle_H| \le \|f\|_{L_2} \|g\|_{B} \; \mbox{ for all }   f \in H,\; g \in B.
\end{equation}

A relation of this kind, while not universal, often holds (especially after utilising the freedom that exists in the choice of $B$), for example it may hold if $\langle f,g\rangle_H$  has a Fourier series definition, or if the inner product in $H$ is the $L_2$ inner product of gradients.  Simple examples are given in the next section.

The first result with the doubled-rate flavor appears to have been obtained by Schaback in~\cite{Sch99}, in the special context of best pointwise approximation by radial basis function interpolation; see also~\cite{HW}.  The paper~\cite{Sch99} is the inspiration for the present doubled-rate result. An earlier progenitor is the Aubin--Nitsche lemma~\cite{JA67,JN68}. 

 A recent paper by Hangelbroek and Rieger~\cite{HR22}, again in the context of radial basis function interpolation, asked a different question, namely whether improved convergence rates can be seen in norms other than $L_2$.  That paper encouraged us to seek the above error bound in the $H$ norm.
 
 The results above are proved in a general Hilbert space context  in Section~\ref{sec:3}, see Theorem~\ref{thm:principal}.
 
A potentially important area of application of the doubling results in this paper is to kernel interpolation.  If $K(\bsx,\bsy)$ for $\bsx, \bsy \in D$ is a given kernel, a  kernel approximation of a function $f$ defined on $D$ is an expression of the form
\[
f_n = \sum_{k=1}^n a_k K(\bst_k, \cdot),
\]
where $\bst_1, \ldots, \bst_n$ is a prescribed set of points in $D$, and $a_1, \ldots, a_n \in \mathbb{R}$.  The kernel approximation is a kernel interpolant if $f_n$ is required to satisfy 
\[
f_n(\bst_j)  =  f(\bst_j), \quad j = 1, \ldots, n.
\]
While many kernels are used in practical applications (Gaussian kernels being perhaps the most popular), of special interest to us is the case in which $K(\cdot,\cdot)$ is the reproducing kernel of a reproducing kernel Hilbert space $H$, because in that case, as shown in Section~\ref{sec:4}, kernel interpolation is equivalent to orthogonal projection with respect to the $H$ norm onto the finite-dimensional subspace defined by
\[
V :=  V_\bst :=\mathrm{span} \{K(\bst_1,\cdot),\ldots,K(\bst_n,\cdot)\}.
\]
A theorem for the doubling phenomenon in a general reproducing kernel Hilbert space $H$ is stated and proved in Section~\ref{sec:4}, see Theorem~\ref{thm:ker_interp}.   In Subsection~\ref{sec:41} we apply the results to recent high-dimensional computations based on kernel interpolation in a reproducing kernel Hilbert space $H$~\cite{KKKNS22,KKS23}, in an application to parametric partial differential equations. In this application the target function is known to be smooth, and hence lies naturally in a smoother normed space $B$.  It was in fact the observation of faster-than-expected convergence in this application that gave birth to the present project.

In Section~\ref{sec:5} we obtain equivalent results for general radial basis function interpolation.  While some radial basis functions are positive definite, and hence in principle covered already by Section~\ref{sec:4}, we here consider general conditionally positive definite basis functions.   In this case $H$ is the so-called ``native space'',  and $B$ is a space of smoother functions introduced by Schaback~\cite{Sch99}.  Improved convergence results for radial basis function interpolation in the $L_2$ and native space norms are given in Theorem~\ref{thm:rbf}.  The proof again rests on the fact that radial basis function interpolation is expressible as orthogonal projection in the space $H$ onto a finite-dimensional subspace.

\section{Examples of suitable spaces $H, B$}\label{sec:2}
 \textbf{Example 1.}  Take $H$ to be the space of real, continuous  functions on $[0,1]$ that have square-integrable first order derivatives. This is a well known  Sobolev space, with standard inner product
\begin{equation*}
\langle f, g \rangle_H := \int_0^1 f(x) g(x) \,{\rm d} x +  \int_0^1 f^\prime(x) g^\prime(x)\,{\rm d} x.
\end{equation*}
Clearly $H$ is continuously embedded in $L_2(0,1)$, since for $f \in H$ we have
\[
\|f\|_{L_2}^2 = \int_0^1 |f(x)|^2 \,\rd x \le \int_0^1 (|f(x)|^2 +|f'(x)|^2)\,{\rm d} x = \|f\|_H^2.
\]

Now define $B$ to be a subspace of $H$ for which also the second derivatives are square-integrable.  On integrating by parts, for $f \in H$ and $g \in B$ we obtain 
\begin{equation*}
\langle f, g \rangle_H := \int_0^1 f(x) g(x)\,{\rm d}x+ f(1)g'(1) - f(0)g'(0) - \int_0^1 f(x) g''(x)\,{\rm d}x,
\end{equation*}
 which with the  Cauchy--Schwarz inequality implies
\[
|\langle f, g \rangle_H| \le \|f\|_{L_2}\|g\|_{L_2} + |f(1)| |g'(1)| + |f(0)| |g'(0)| +\|f\|_{L_2} \|g''\|_{L_2}.
\]
For \eqref{eq:norms_relation} to be satisfied the two central terms must vanish, since $f(0)$  and $f(1)$ are not bounded functionals in $L_2(0,1)$, thus we must define $B$ to be the subspace of the Sobolev space of order $2$  for which $g'(1) = g'(0) = 0$.  We then have
\[
|\langle f, g \rangle_H| \le \|f\|_{L_2}\left(\|g\|_{L_2} + \|g''\|_{L_2}\right) \le \|f\|_{L_2}\left(2(\|g\|_{L_2}^2 + \|g''\|_{L_2}^2)\right)^{1/2},
\]
thus \eqref{eq:norms_relation} is satisfied if the norm in $B$ is defined by 
\[
\|g\|_B^2 := 2\left(\|g\|_{L_2}^2 +\|g'\|_{L_2}^2 +\|g''\|_{L_2}^2\right).
\] 
 Clearly $B$ is continuously embedded in $H$, since 
 \[
 \|g\|_H^2 = \|g\|_{L_2}^2 + \|g'\|_{L_2}^2 \le \tfrac{1}{2}\|g\|_B^2 \quad \text{for all}~g \in B.
 \]

\textbf{Example 2.}  The next example shows that some care may be required in defining the subspace $B$.  We take $H$ to be the same space as in the first example. Then another possible inner product is
\begin{equation}\label{eq:inner_prod2}
\langle f, g \rangle_H := f(0) g(0) + \int_0^1 f^\prime(x) g^\prime(x)\,{\rm d}x.
\end{equation} 
Assuming that $g$ has square-integrable second derivatives, we may integrate by parts to obtain
\[
\langle f, g \rangle_H := f(0) g(0) + f(1)g'(1) - f(0) g'(0) -\int_0^1 f(x) g''(x)\,{\rm d} x.
\]
Equation \eqref{eq:norms_relation} is then clearly satisfied if we define $B \subset H$ by
\begin{equation}
B = \{g \in H: g'' \in L_2(0,1) \mbox{ and } g(0) = g'(0) = g'(1) =0\},
\label{eq:B2}
\end{equation}
and take $\|g\|_B := \|g''\|_{L_2}$, which again is easily seen to be a norm in $B$.

We demonstrate in Section~\ref{sec:4} that $H$ is continuously embedded in $L_2(0,1)$ and likewise that $B$ is continuously embedded in $H$.

 \textbf{Example 3.}  Take $H$ to be the space of real, continuous  functions on $[0,1]$  which have square-integrable first order derivatives and which are periodic in the sense that $f(0)=f(1)$ for $f \in H$. This is a Hilbert space with inner product 
\begin{equation*}
\langle f, g \rangle_H :=\left(\int_0^1 f(x)\,{\rm d} x\right)\left(\int_0^1 g(x)\,{\rm d} x\right) +  \int_0^1 f^\prime(x) g^\prime(x)\,{\rm d} x.
\end{equation*}
By Wirtinger's inequality (see, e.g.,~\cite[Theorem~258]{HLP}), we obtain
$$
\bigg\|f-\int_0^1 f(x)\,{\rm d}x\bigg\|_{L_2}^2\leq \frac{1}{(2\pi)^2}\|f'\|_{L_2}^2\quad\text{for all}~f\in H,
$$
which yields that $\|f\|_{L_2}^2\leq \big(\int_0^1f(x)\,{\rm d}x\big)^2+(2\pi)^{-2}\|f'\|_{L_2}^2\leq \|f\|_H^2$ for all $f\in H$. Therefore $H$ is continuously embedded in $L_2(0,1)$.

Now define $B$ to be a subspace of $H$, one in which the second derivatives are square-integrable, and for which the first derivatives at $0$ and $1$ are equal.  On integrating by parts, and using the periodicity, we obtain 
\begin{align*}
\langle f, g \rangle_H &:=\left(\int_0^1 f(x)\,{\rm d} x\right)\left(\int_0^1 g(x)\,{\rm d} x\right) -  \int_0^1 f(x) g''(x)\,{\rm d} x\\
&=\int_0^1 f(x) \left( \int_0^1 g(y)\,\rd y - g''(x) \right) {\rm d} x,
\end{align*}
implying with the Cauchy--Schwarz inequality
\begin{align*}
|\langle f, g \rangle_H| &\le \|f\|_{L_2}\left(\int_0^1 \left(\int_0^1  g(y)\,\rd y -g''(x)\right)^2{\rm d} x\right)^{1/2}\\
 &= \|f\|_{L_2}\left(\left(\int_0^1 g(y)\,\rd y\right)^2 + \int_0^1 |g''(x)|^2 \,{\rm d} x\right)^{1/2},
\end{align*}
where in the last step we used the fact that $\int_0^1 g''(x) \,\rd x= g'(1) - g'(0) = 0$.  The last equation conforms with \eqref{eq:norms_relation} if the last factor is used to define the norm $\|g\|_B:=\big(\big(\int_0^1g(x)\,{\rm d}x\big)^2+\int_0^1|g''(x)|^2\,{\rm d}x\big)^{1/2}$ in $B$.

To see that $B$ is continuously embedded in $H$, we can argue in a similar way as above to obtain
$$
\|g'\|_{L_2}^2=\bigg\|g'-\int_0^1g'(x)\,{\rm d}x\bigg\|_{L_2}^2\leq \frac{1}{(2\pi)^2}\|g''\|_{L_2}^2\quad\text{for all}~g\in B,
$$
since $\int_0^1g'(x)\,{\rm d}x=g(1)-g(0)=0$. This immediately yields that $\|g\|_H\leq \|g\|_B$ for all $g\in B$, as desired.

\textbf{Example 4.} Let $D\subset \mathbb R^d$, $d\in\{1,2,3\}$, be a nonempty, bounded domain with Lipschitz boundary. We check the relation~\eqref{eq:norms_relation} for the Sobolev spaces $H=H_0^1(D)$ and $B=H^2(D)\cap H_0^1(D)$, which are natural Hilbert spaces to analyse variational solutions to elliptic partial differential equations of the form
$$
\begin{cases}
-\nabla\cdot (a(\bsx)\nabla u(\bsx))=z(\bsx)&\text{for}~\bsx\in D,\\
u|_{\partial D}=0,&
\end{cases}
$$
where the functions $a,z\!:D\to\mathbb R$ are called the diffusion coefficient and source term, respectively. Standard elliptic regularity theory states that if the domain $D$ is convex, the diffusion coefficient $a$ is Lipschitz continuous, and the source term $z\in L_2(D)$, then the solution $u$ belongs to $B$. 

We can define $\langle f,g\rangle_H:=\langle f,g\rangle_{L_2(D)}+\langle \nabla f,\nabla g\rangle_{L_2(D)}$ for $f,g\in H$ so that
$$
\|f\|_{L_2(D)}\leq \|f\|_H\quad\text{for all}~f\in H,
$$
which implies the continuous embedding of $H$ in $L_2(D)$.

Let $f\in H$ and $g\in B$. We obtain by Green's first identity that
\begin{align*}
\langle f,g\rangle_H=\int_D f(\bsx)g(\bsx)\,{\rm d}\bsx-\int_D f(\bsx)\Delta g(\bsx)\,{\rm d}\bsx.
\end{align*}
This allows us to estimate
\begin{align*}
|\langle f,g\rangle_H|&\leq \|f\|_{L_2(D)}(\|g\|_{L_2(D)}+\|\Delta g\|_{L_2(D)}),
\end{align*}
where the second factor specifies a norm $\|g\|_B:=\|g\|_{L_2(D)}+\|\Delta g\|_{L_2(D)}$ in $B$.  We note that this is not a Hilbert space norm, serving as a reminder that the requirement on $B$ is that it be a normed vector space continuously embedded in $H$, not necessarily a Hilbert space. 

It follows from the Poincar\'e inequality that there exists a constant $C>0$ depending on the domain $D$ such that
$$
\|f\|_{L_2(D)}\leq C\|\nabla f\|_{L_2(D)}\quad\text{for all}~f\in H.
$$
Letting $g\in B$, we obtain by Green's first identity that
\begin{align*}
\|\nabla g\|_{L_2(D)}^2&=-\int_D g(\bsx)\Delta g(\bsx)\,{\rm d}\bsx\\
&\leq \|g\|_{L_2(D)}\|\Delta g\|_{L_2(D)}\\
&\leq C\|\nabla g\|_{L_2(D)}\|\Delta g\|_{L_2(D)}.
\end{align*}
This yields $\|\nabla g\|_{L_2(D)}\leq C\|\Delta g\|_{L_2(D)}$ and thus $\|g\|_H\leq (1+C^2)^{1/2}\|g\|_B$ for all $g\in B$. Hence $B$ is continuously embedded in $H$.

 \section{The main theorem}\label{sec:3}
 The following theorem captures the two improved error bounds described in the Introduction.  However, we here state the results in a more general way, replacing $c n^{-\kappa}$ by $M$, where $M>0$ depends on $V$ and $H$.
 \begin{theorem}\label{thm:principal}
For a domain $D\subseteq \mathbb{R}^d$, let $H$ be a Hilbert space and $B$ a normed space of real-valued functions on $D$ such that $H$ is continuously embedded in $L_2(D)$ and $B$ is continuously embedded in $H$, and such that \eqref{eq:norms_relation} holds. Further,  let $P$ be the $H$-orthogonal projection operator onto a finite-dimensional space $V\subset H$.  Assume that for some $M := M(V, H)>0$ we have
\begin{equation}\label{eq:orthog_error0}
 \|f - P f\|_{L_2} \le M \|f\|_H \; \mbox{ for all }    f \in H.  
 \end{equation}
 Then for all $g$ in $B$ we have 
 \[
 \|g - P g\|_{L_2} \le M^2 \|g\|_B,  
 \]
 and
 \[
 \|g - P g\|_H \le M \|g\|_B.  
 \]
\end{theorem}
  
  \begin{proof}
 We first apply \eqref{eq:orthog_error0} to $f = g - P g$, noting that it belongs to $H$ since $g \in B \subset H$ and $Pg \in H$.  Using  the idempotent property of $P$ we have $P(g - Pg) = Pg - Pg  = 0$, thus by \eqref{eq:orthog_error0} we obtain
\begin{equation}\label{eq:H0error_bdd_by_2norm}
\|g -Pg\|_{L_2}  = \|(g-Pg) - P(g - Pg)\|_{L_2}
\le M \|g - P g\|_H.  
 \end{equation}
 Now 
 \begin{align*}
  \|g - P g\|_H^2
  &= \langle g - P g,g - P g\rangle_H\\
  &=  \langle g - P g,g\rangle_H,
  \end{align*}
since by definition $g - P g$ is $H$-orthogonal to every element of $V$.  It follows using  \eqref{eq:norms_relation} with $f := g - Pg$ that
\begin{equation}\label{eq:H1error_bdd_by_2norm}
\|g - Pg\|_H^2 \le \|g - P g\|_{L_2} \; \|g\|_B.
\end{equation}

From \eqref{eq:H0error_bdd_by_2norm} and \eqref{eq:H1error_bdd_by_2norm} we now obtain
\begin{align*}
\|g - P g\|_{L_2}^2
&\le M^2\|g - P g\|^2_H\\
&\le M^2 \|g - P g\|_{L_2} \|g\|_B.
\end{align*}
The first result now follows by cancelling $ \|g - P g\|_{L_2}$.

The second result, in squared form, then follows by substituting the first result into~\eqref{eq:H1error_bdd_by_2norm}.
 \end{proof}

\noindent \emph{Remark.} A similar argument to that above  was used by Wendland in~\cite[Theorem 11.23]{HW}.

\section{Kernel interpolation in reproducing kernel Hilbert spaces}\label{sec:4}
In this section we take  $H$ to be a reproducing kernel Hilbert space of real-valued functions defined on a domain $D$.  We show that the general theorem, Theorem~\ref{thm:principal}, has direct application to interpolation with respect to a set of kernels anchored at the interpolation points.

A reproducing  kernel Hilbert space (RKHS) is a Hilbert space in which point evaluation is bounded.  By the Riesz representation theorem, for every RKHS there exists a kernel $K(\cdot,\cdot)$, with the properties  
\begin{equation}\label{eq:KinH}
K(\bsy, \cdot)\in H \quad \mbox{  for all }\bsy \in D,
\end{equation}
 and 
 \begin{equation}\label{eq:reprod_prop}
 \langle K(\bsy, \cdot), f\rangle_H = f(\bsy)  \quad \mbox{ for all } f \in H \mbox{ and } \bsy \in D,
 \end{equation}
 the latter being the reproducing property.
 
 It is easily seen, using
 \[
 K(\bsy, \bsx) =   \langle K(\bsy, \cdot), K(\bsx, \cdot)\rangle _H,
 \] 
 that the reproducing  kernel is symmetric,
 \begin{equation*}
 K(\bsy, \bsx) = K(\bsx, \bsy)\quad  \mbox{ for all }\bsy, \bsx \in D.
\end{equation*}
The kernel is unique (but not necessarily known).  The reproducing kernel is positive definite, in the sense that
\begin{equation}\label{eq:pos_def}
\sum_{j = 1}^N  \sum_{k = 1}^N  b_j K(\bst_j, \bst_k) b_k \ge 0 
\end{equation}
for all $N \in \mathbb{N}$, all choices of points $\bst_j $ and all real  $b_j, j = 1,\ldots, N$, with equality if and only if all $b_j$ vanish.  The reverse is also true, that to every positive definite kernel there is a reproducing kernel Hilbert space $H$.  For further information about reproducing kernel Hilbert spaces, see~\cite{Aron} or the more recent source~\cite[Chapter 10]{HW}.  

\vspace{.2cm}
 \textbf{Example 2 continued.} Recall that $H$ is the space of continuous real-valued functions $f$ defined on $[0,1]$ whose first  derivatives are square-integrable, and that the inner product in $H$ is given by \eqref{eq:inner_prod2}.  It can easily be seen using that inner product  that the kernel 
 \[
 K(x,y) : = 1+ \min(x,y)
 \]
satisfies both \eqref{eq:KinH} and  \eqref{eq:reprod_prop} for all $f \in H$, and that  $K(x,y)$ is therefore the reproducing kernel of $H$.
 
We can now show that $H$ is continuously embedded in $L_2(0,1)$ and that $B$ defined by \eqref{eq:B2} is continuously embedded in $H$. By  using the reproducing property $f(x)=\langle f,K(\cdot,x)\rangle_H$ for all $f\in H$ and $x\in [0,1]$, we obtain
\begin{align}
\|f\|_{L_2}^2&=\int_0^1 |f(x)|^2\,{\rm d}x = \int_0^1 \langle f,K(\cdot,x)\rangle_H^2\,{\rm d}x \leq \int_0^1 \|f\|_H^2 \|K(x,\cdot)\|_H^2 \,{\rm d} x \nonumber \\ &= \|f\|_H^2 \int_0^1 K(x,x)\, {\rm d} x = \frac32\|f\|_H^2\quad \text{for all}~f\in H,
\label{eq:Kemb}
\end{align}
which means that $H$ is continuously embedded in $L_2(0,1)$. Let $g\in B$. To show the continuous embedding of $B$ in $H$, we integrate by parts and use  \eqref{eq:B2} to obtain
\begin{align*}
\|g\|_H^2&=g(0)^2+\int_0^1 |g'(x)|^2\,{\rm d}x\\
&=g(1)g'(1)-g(0)g'(0)-\int_0^1 g(x)g''(x)\,{\rm d}x\\
&\leq \|g\|_{L_2}\|g''\|_{L_2}\leq \sqrt{\frac32}\|g\|_H\|g\|_B,
\end{align*}
where the final step follows from~\eqref{eq:Kemb}. Therefore $\|g\|_H\leq \sqrt{\frac32}\|g\|_B$, showing the continuous embedding  of $B$ in $H$.

For the general case, a kernel interpolant of a function $f \in H$ is an approximation of the form 
 \begin{equation}\label{eq:ker_int}
 f_{n}(\bsy) :=  \sum_{k = 1}^n a_k K(\bst_k, \bsy),
 \end{equation}
where $\bst_1, \ldots, \bst_n$ are distinct points in $D$, and the coefficients $a_k$ are  determined by the interpolation condition
\begin{equation}\label{eq:interp_cond}
f_{n}(\bst_k) = f(\bst_k),\quad k = 1,\ldots, n.
\end{equation}
The linear system necessarily has a unique solution $(a_k)_{k=1}^n$, because of the positive definite property \eqref{eq:pos_def} of the reproducing kernel.

Kernel methods have a long history, tracing back to Grace Wahba's seminal work on splines~\cite{GW}. They have an important role in radial basis functions see~\cite{HW}, and have recently been used in high dimensional approximation, see for example~\cite{GRZ19,KKKNS22,KKS23,ZKH09,ZLH06}.

The kernel interpolant fits the scenario of the Introduction and Theorem~\ref{thm:principal} as long as there exists a suitable subspace $B \subset H$ such that  \eqref{eq:norms_relation} is satisfied.  The first step is to define 
\[
V:= \mathrm{span} \{K(\bst_1, \cdot), K(\bst_2, \cdot), \ldots, K(\bst_n, \cdot)\},
\]
which by property \eqref{eq:KinH} is a finite-dimensional subspace of $H$.

The crucial point is the observation that, for $H$ a reproducing kernel Hilbert space, the kernel interpolant  $f_{n}$ is the $H$-orthogonal projection of $f$ on $V\subset H$.  This is an immediate consequence of  the reproducing property \eqref{eq:reprod_prop} together with the interpolation condition \eqref{eq:interp_cond}, because for all $f \in H$ we have
 \[
 \langle f - f_{n}, K(\bst_k, \cdot)\rangle_{H} =  f(\bst_k) - f_{n}(\bst_k) = 0,
 \] 
 from which it follows that $f - f_n$ is $H$-orthogonal to every element of $V$.
 
The following proposition thus is an immediate consequence of Theorem~\ref{thm:principal}. 

\begin{theorem}\label{thm:ker_interp}
 Let $H$ be a reproducing kernel Hilbert space continuously embedded in $L_2(D)$ and with kernel $K$, and let  $B$ be a normed space continuously embedded in  $H$ such that  \eqref{eq:norms_relation} holds.  For $\bst_1, \ldots, \bst_n$ a set of distinct points in $D$, let $f_{n}$ be the kernel interpolant defined by \eqref{eq:ker_int} and \eqref{eq:interp_cond}.  Assume that for some $c>0$ and some $\kappa> 0$ we have
\begin{equation*}
 \|f - f_{n}\|_{{L_2}(D)} \le c n^{-\kappa} \|f\|_H \; \mbox{ for all }    f \in H.  
 \end{equation*}
 Then for all $g$ in $B \subset H$, with kernel interpolant $g_n$ using the same points $\bst_1, \ldots,\bst_n$, we have
 \[
 \|g - g_n\|_{{L_2}(D)} \le c^2   n^{-2\kappa}\|g\|_{B},  
 \]
 and 
 \[
 \|g - g_n\|_H \le  c n^{-\kappa}\|g\|_{B}.  
 \]
\end{theorem}

\subsection{A concrete application to kernel interpolation}\label{sec:41}

In recent years there has been increasing interest in the use of kernel interpolation in challenging large-scale computations.  In particular, in the papers~\cite{KKKNS22,KKS23}  kernel interpolation was used for the solution of a parametric partial differential equation with a very large number of parameters -- in the first paper detailed calculations were carried out with 10 and 100 independent parameters, while in the second paper the number of parameters was extended to 1,000.  

Our purpose in this subsection is not to review the substance of those papers, but rather to present the high-dimensional setting encountered there, and to identify (in the notation of the present paper) the function spaces $H$ and $B$, along with the reproducing kernel $K$ of $H$.  We will be then be able to see that Theorem~\ref{thm:ker_interp} applies directly to the computations in those papers, and hence be able to explain the higher than expected rates of convergence reported there.

In the papers~\cite{KKKNS22,KKS23} the space $D$ in which the parameters lie is $[0,1]^d$, with $d$ typically large. The function space $H$ is a Sobolev space of  dominating mixed smoothness $\alpha \in \mathbb{N}$.  It is a space of periodic functions (that is, each function in $H$ is $1$-periodic with respect to each component of the argument),
which is a natural high-dimensional generalisation of Example 3 in Section~\ref{sec:2}.  However, it proves to be simpler to use a Fourier representation to describe the space and its norm.

 Specifically, for arbitrary $\alpha>1/2$ we define $f$ to be in the space $H := H_{\alpha, \bsgamma}$ if 
 \begin{equation}\label{eq:norm_ag}
 \|f\|_{\alpha, \bsgamma}^2 := \sum_{\bsh\in \mathbb{Z}^d} r_{\alpha, \bsgamma}(\bsh)|\widehat{f}(\bsh)|^2 < \infty,
 \end{equation}
 where 
 \[
 \widehat{f}(\bsh):=\int_{[0,1]^d}f(\bsy)\mathrm{e}^{-2\pi \mathrm{i} \bsh\cdot \bsy}\,\rd \bsy,
 \]
 and where for $\bsh \in \mathbb{Z}^d$ 
 \begin{equation}\label{eq:r_ag}
 r_{\alpha, \bsgamma}(\bsh) := \frac{1}{\gamma_{\mathrm{supp}(\bsh)}}\prod_{j \in \mathrm{supp}(\bsh)}|h_j|^{2\alpha},
 \end{equation}
 with $\mathrm{supp}(\bsh) := \{j \in \{1,\ldots,d\}:h_j \ne 0\}$.  The inner product in $H$ is
\begin{equation}\label{eq:inner_ag}
 \langle f,g \rangle_{\alpha, \bsgamma} := \sum_{\bsh\in \mathbb{Z}^d} r_{\alpha, \bsgamma}(\bsh)\widehat{f}(\bsh)\overline{\widehat{g}(\bsh)} < \infty.
 \end{equation}

 (To avoid confusion we note that the present definition of the parameter $\alpha$ follows~\cite{KKS23}, not~\cite {KKKNS22}.  The two definitions differ by a factor of $2$.)
 
 The above definition of the norm in $H$  incorporates ``weights'' $\gamma_\setu$, which are prescribed positive numbers, one for each subset $\setu\subseteq \{1, \ldots, d\}$.  It is by now well known that if the weights are omitted (that is, by setting each $\gamma_\setu$ equal to $1$), then most problems become ``intractable'' as $d \to \infty$.  In the paper~\cite{KKKNS22} much effort is devoted to choosing weights that ensure that the interpolation error is  bounded as $d\to\infty$, but that topic is outside the scope of this paper.  For the present purposes we can assume that the weights are all given.
 
It is easily seen (by verifying the properties \eqref{eq:KinH} and \eqref{eq:reprod_prop}) that $H$ is a reproducing kernel Hilbert space, with reproducing kernel
\begin{equation*}
K(\bsy, \bsx)= \sum_{\setu \subseteq{\{1, \ldots, d\}}}\gamma_\setu \prod_{k\in \setu}\eta_\alpha(y_k, x_k) = \sum_{\bsh \in \mathbb{Z}^d}
\frac{e^{2\pi i \bsh \cdot (\bsy - \bsx)}}{r_{\alpha, \bsgamma}(\bsh)}, 
\end{equation*}
where
\[
\eta_\alpha(y,x) := \sum_{h \ne 0} \frac{e^{2\pi i h (y - x)}}{|h|^{2\alpha}},
\] 
which can be expressed explicitly as a multiple of  $B_{2\alpha}(\{y - x\})$, where $B_\tau$ denotes the Bernoulli polynomial of degree $\tau$, and the braces indicate taking the fractional part of the argument.
 
 To apply Theorem~\ref{thm:ker_interp} it remains to identify a suitable embedded subset $B \subset H$.  The following proposition tells us that we may take $B:= H_{2\alpha, \gamma^2}$ which is a weighted Hilbert space of dominating mixed smoothness of order  $2\alpha$, so having twice the smoothness parameter of $H$, and with each weight $\gamma_\setu$ replaced by its square.  
 \begin{proposition}
 With $H_{\alpha, \gamma}$ defined as above, there holds
 \begin{equation*}
 |\langle f,g \rangle_{\alpha, \bsgamma}| \le \|f\|_{L_2} \|g\|_{2\alpha, \bsgamma^2} \quad \mbox{ for all }  f \in H_{\alpha,\bsgamma},\; g \in H_{2\alpha,\bsgamma^2},
 \end{equation*}
 where the inner product and norm are defined by \eqref{eq:inner_ag} and \eqref{eq:norm_ag}, and $\bsgamma^2  :=  (\bsgamma_\setu^2)_{\setu\subseteq\{1,\ldots,d\}}$.
 \end{proposition}

\begin{proof}
From \eqref{eq:inner_ag} and the Cauchy--Schwarz inequality we have 
\begin{align*}
\left| \langle f, g \rangle_{\alpha,\bsgamma}\right|
&= \left| \sum_{\bsh \in\mathbb{Z}^d}  \widehat{f} (\bsh) \overline{\widehat{g}(\bsh)} \,   r_{\alpha,\bsgamma}(\bsh) \right|\\
&\le \left(\sum_{\bsh\in\mathbb{Z}^d} |\widehat{f}(\bsh)|^2  \right)^{1/2}   \left( \sum_{\bsh\in\mathbb{Z}^d}  |\widehat{g}(\bsh)|^2  r_{\alpha, \gamma}(\bsh)^2\right)^{1/2}\\
&=  || f ||_{L_2}    || g ||_{2\alpha, \bsgamma^2},
\end{align*}
where we used $r_{\alpha, \bsgamma}(\bsh)^2 = r_{2\alpha, \bsgamma^2}(\bsh)$, which follows immediately from \eqref{eq:r_ag}.
\end{proof}

The following corollary now follows from the proposition above and Theorem~\ref{thm:ker_interp} with $H$ replaced by $H_{\alpha, \bsgamma}$ and $B$ replaced by $H_{2\alpha, \bsgamma^2}$.  It is stated in a form that shows directly the doubling of the convergence rate.

\begin{corollary}\label{cor:weighted_ker_interp}
 For given $\alpha >0$, let $H_{\alpha, \bsgamma}$ be the reproducing kernel Hilbert space constructed as above. With $K_{\alpha,\bsgamma}$ denoting the corresponding reproducing kernel, let $f_{\alpha,\bsgamma, n}$ be the kernel interpolant for this kernel that satisfies  \eqref{eq:ker_int} and \eqref{eq:interp_cond}.  Assume that for some $c>0$ and some $\kappa> 0$ we have
\begin{equation}\label{eq:wtd_error}
 \|f - f_{\alpha,\bsgamma,n}\|_{L_2} \le c n^{-\kappa} \|f\|_{\alpha, \bsgamma} \; \mbox{ for all }    f \in H_{\alpha,\bsgamma}.  
 \end{equation}
 Then for all $g\in H_{2\alpha, \bsgamma^2}\subset H_{\alpha,\bsgamma}$ we have 
 \[
 \|g - g_{\alpha, \bsgamma, n}\|_{L_2} \le c^2   n^{-2\kappa}\|g\|_{2\alpha, \bsgamma^2},  
 \]
 and 
 \[
 \|g - g_{\alpha, \bsgamma, n}\|_{\alpha, \bsgamma} \le  c n^{-\kappa}\|g\|_{2\alpha, \bsgamma^2}.  
 \]
\end{corollary} 
  
A final comment is that in the papers~\cite{KKKNS22,KKS23} the hypothesis~\eqref{eq:wtd_error} was established with $\kappa = \tfrac{1}{2}\alpha-\delta$ for arbitrary $\delta > 0$, a result known (because of the restricted nature of the interpolation points used there) to be not improvable for general $f\in H_{\alpha, \bsgamma}$.  Since the target function $f$  in those papers is known to be infinitely differentiable, and hence lying in $H_{2\alpha, \bsgamma^2}$, the results in those papers can now be supplemented by the conclusions of this corollary.  In particular, the asymptotic rate of convergence in the $L_2$ norm is now proven to be (at least) $n^{-\alpha}$ rather than merely $n^{-\alpha/2}$,  while the asymptotic convergence rate in the $H_{\alpha, \bsgamma}$ norm (for which no previous estimate was available) is now known to be $n^{-\alpha/2}$.
We note that Corollary~\ref{cor:weighted_ker_interp} also gives values for the ``constants'' in the new bounds.  Those constants may be very large, and in the context of~\cite{KKKNS22,KKS23} may also depend strongly on $d$ through the appearance of the $H_{2\alpha, \bsgamma^2}$ norm.

\textbf{Example 5.} 
We consider the rate doubling phenomenon numerically for the Korobov spaces $H:=H_{1,\boldsymbol\gamma}$ and $B:=H_{2,\boldsymbol\gamma^2}$ with positive weights $\boldsymbol\gamma=(\gamma_{\setu})_{\setu\subseteq\{1,\ldots,d\}}$  when $d=10$.  These spaces are equipped with norms given by~\eqref{eq:norm_ag} for $\alpha=1$ and $\alpha=2$, respectively. We consider here the product weights given by $\gamma_{\setu}=\prod_{j\in \setu}\gamma_j$, where we set $\gamma_j=0.5j^{-1.5}$ for $j\in\{1,\ldots,d\}$.

The kernel interpolant is constructed using the reproducing kernel $K_1$ of $H$, and the point set is given by
$$
\bst_\ell = \bigg\{\frac{\ell \boldsymbol z}{n}\bigg\},\quad \ell\in\{0,\ldots,n-1\},
$$
where the generating vector $\boldsymbol z\in\mathbb N^d$ is obtained by using the component-by-component (CBC) algorithm described in~\cite{ckns21} for the product weights $\boldsymbol\gamma=(\gamma_{\setu})_{\setu\subseteq\{1,\ldots,d\}}$ in $H$.

We compute the $L_2$ approximation errors for 
\begin{align*}
f(\bsx)=K_1(\bsx,e^{-1}\mathbf 1),\quad \bsx\in[0,1]^d,
\end{align*}
where $\mathbf 1\in\mathbb R^d$ is a vector of ones and $e$ denotes Euler's number, as well as
$$
g(\bsx)=K_2(\bsx,e^{-1}\mathbf 1),\quad \bsx\in[0,1]^d,
$$
where $K_2$ denotes the reproducing kernel of $B$ corresponding to the squared weights $\boldsymbol\gamma^2$. Now $f\in H$ and $g\in B$, but $f\not\in B$. In this setting, it is known that the kernel interpolant satisfies~\eqref{eq:first_ass} with $\kappa=\alpha/2$, see~\cite[Theorem~3.2]{KKKNS22}. Thus the expected convergence rates are $-0.5$ and $-1.0$ for $f$ and $g$, respectively.

The $L_2$ approximation errors $\|f-f_n\|_{L_2}$ and $\|g-g_n\|_{L_2}$ were obtained by using a $1000$-point Monte Carlo approximation. The numerical results are displayed in Figure~\ref{fig:numex}. As expected, the function $f$ belonging to $H$, but not $B$, exhibits a convergence rate close to $-0.5$ while the function $g$ has a convergence rate close to $-1.0$.

\begin{figure}[!t]
\begin{center}
\includegraphics[width=.74\textwidth]{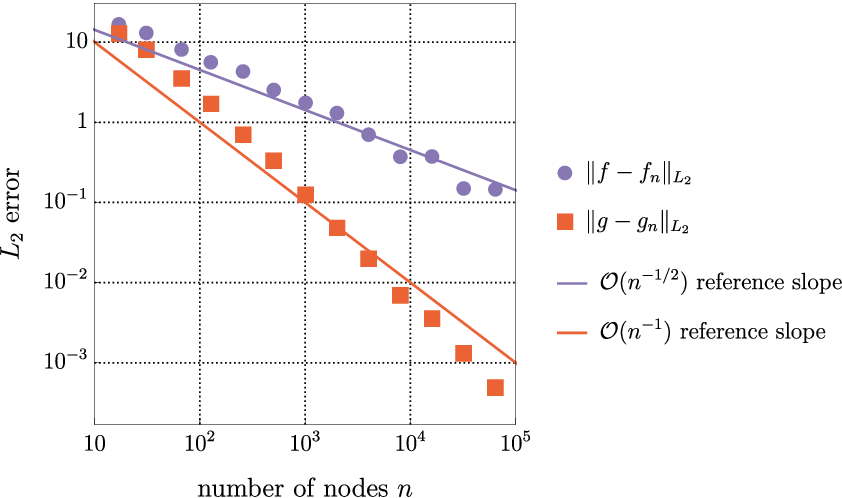}
\end{center}
\caption{Numerical $L_2$ approximation errors for the functions $f$ and $g$  considered in Example~5.}\label{fig:numex}
\end{figure}

\section{Radial basis function interpolation}\label{sec:5} 
 
 Radial basis functions provide an important and well studied method for approximating a continuous function given its values at scattered data points, see~\cite{HW}.  Our aim here is not to provide a guide to radial basis functions, but rather to demonstrate how the general Theorem~\ref{thm:principal} can be extended to radial basis function interpolation.  In general we follow the conventions of the monograph~\cite{HW}.
 
\begin{definition}\rm 
A real, even, continuous function $\Phi:\mathbb{R}^d \to \mathbb{R}$  
is said to be \emph{conditionally positive semi-definite} of order $m\ge 0$ if for all $n \in \mathbb{N}$ and all choices of points $\bst_1, \ldots,\bst_n \in \mathbb{R}^d$ we have
\begin{equation}\label{eq:conditional_1}
\sum_{k = 1}^n \sum_{j=1}^n a_k \Phi(\bst_k - \bst_j) a_j \ge 0
\end{equation}
for all $\bsa  = (a_1, \ldots, a_n) \in \mathbb{R}^n$ that satisfy
\begin{equation}\label{eq:conditional_2}
\sum_{k=1}^n a_k p(\bst_k) = 0 \mbox{ for all polynomials } p \mbox{ of degree less than } m.
\end{equation}
It is said to be \emph{conditionally positive definite of order $m$} if equality in~\eqref{eq:conditional_1} holds only for $\bsa =\bszero$.  We note that the condition~\eqref{eq:conditional_2} becomes increasingly restrictive on the coefficients $a_k$ as $m$ increases.
\end{definition}

 We will assume that the order $m$ in the definition takes its smallest possible value. 
 
Some popular choices for $\Phi$ in the context of radial basis function interpolation include Gaussians, multiquadrics, radial powers, and thin-plate splines. New stable bases for conditionally positive definite kernel-based spaces have been recently developed in~\cite{MDME24}.
  
 \vspace{.2cm}
 \textbf{Example 6.}  
 The thin-plate spline $\Phi(\bsy):= \|\bsy\|_{\ell_2}^2 \log \|\bsy\|_{\ell_2}$ is conditionally positive definite of order $2$.
 
 \smallskip
 
In the special case $m = 0$ the condition~\eqref{eq:conditional_2} is vacuous, and a conditionally positive definite function $\Phi$  yields  
a kernel $\Phi(\bsy, \bsx):= \Phi(\bsy - \bsx)$ that is positive definite in 
the sense of~\eqref{eq:pos_def}.  Some important radial basis functions are known to be  the conditionally positive definite of order $m = 0$.  These include Gaussians, inverse multiquadrics and Wendland's compactly supported radial basis functions~\cite{HW}.   Here, however, we shall concentrate on the general case of order $m \ge 0$.

While all of the basis functions used in practice are radial, in the sense of depending, like the thin plate spline, on the $\ell_2$ norm of the argument, we shall not make any such general assumption in this section, though in line with convention we shall continue to speak of radial basis functions, or RBFs.

 We now suppose that a continuous function $f$ is defined on a bounded domain  $\Omega \subset \mathbb{R}^d$  with Lipschitz boundary.    We assume that $f$ has known values at $n$ distinct points $\bst_1,\ldots, \bst_n \in \Omega$, which we wish to interpolate with the help of the conditional positive definite $\Phi$ of order $m \ge 0$.
 
 There is an essential condition on the interpolation point set, namely that it be $\Pi_{m-1}$-unisolvent.  Here $\Pi_{m-1}$ denotes the set of polynomials on $\mathbb{R}^d$ of degree less than  or equal to $m-1$, and $\Pi_{m-1}$-unisolvency means that the only polynomial of degree less than $m$ that vanishes at all the points $\bst_1,\ldots,\bst_n$ is the zero function.  If $m = 0$ the set $\Pi_{m-1}$ is empty, and the unisolvency condition is not needed.

 The radial basis function interpolant of $f$ using the point set $\bst$ has the form
 \begin{equation}\label{eq:interpolant_rbf}
 s_{f,\bst}(\bsy) =  f_{\bst}(\bsy) + p_\bst(\bsy),
 \end{equation}
 where, with the notation $\Phi(\bsx, \bsy):= \Phi(\bsx - \bsy)$, 
 \begin{equation}\label{eq:fa_part}
 f_{\bst}(\bsy) = \sum_{k = 1}^ n  a_k \Phi(\bst_k ,\bsy ),
 \end{equation}
 with $\bsa = (a_1, \ldots, a_n) \in \bbR^n$ constrained by~\eqref{eq:conditional_2}, and $p_\bst \in \Pi_{m-1}$.
 The function $s_{f,\bst}$, and its components $f_{\bst}$ and $p_\bst$, are then uniquely determined by the interpolation condition
 \begin{equation}\label{eq:interp_cond_rbf}
 s_{f,\bst}(\bst_k) = f(\bst_k) \quad\mbox{for } k = 1,\ldots,n,
 \end{equation}
with the uniqueness following from the fact that 
\begin{equation}\label{eq:unique}
 \sum_{k = 1}^ n  a_k \Phi(\bst_k , \cdot ) +p_\bst(\cdot)= 0,
 \end{equation}
with $\boldsymbol a$ constrained by~\eqref{eq:conditional_2},
 implies  
\[
 \sum_{j=1}^n\sum_{k = 1}^ n  a_k \Phi(\bst_k , \bst_j)a_j +\sum_{j = 1}^n a_j p_\bst(\bst_j)= 0,
 \]
 in which the last term vanishes by~\eqref{eq:conditional_2}, and hence
 \[
  \sum_{j=1}^n\sum_{k = 1}^ n  a_k \Phi(\bst_k , \bst_j)a_j  = 0,
  \]
which in turn implies $a_1 =\cdots =a_n =0$, after which~\eqref{eq:unique} yields $p_\bst = 0$ and hence $s_{f ,\bst}= 0$.

With that foundation, we now turn to defining $H$ and $B$ for RBF interpolation.   We first define the linear space generated by all functions of the form~\eqref{eq:fa_part}, 
\[
F_\Phi :=\left\{\sum_{k = 1}^n a_k \Phi(\bst_k, \cdot):  \,\bst_1, \ldots, \bst_n \in \Omega,~\bsa \in \bbR^n \mbox{ satisfying  \eqref{eq:conditional_2},  and } n\in\mathbb{N}  \right\}.
\]  
An inner product on $F_\Phi$ can be defined  by
\begin{equation} \label{eq:inner_F}
 \langle f_{\bst}, g_{\bsu}\rangle_{F_\Phi} := \sum_{k=1}^n \sum_{j=1}^\ell\ a_k \Phi(\bst_k, \bsu_j) b_j, 
 \end{equation}
 where $n, \ell \in \mathbb{N}$, 
  \begin{equation}\label{eq:ftgu}
  f_\bst = \sum_{k=1}^n  a_k \Phi(\bst_k,\cdot), \quad  g_\bsu = \sum_{j=1}^\ell  b_j \Phi(\bsu_j,\cdot), \quad \bst_1,\ldots,\
  \bst_n,\bsu_1, \ldots, \bsu_\ell \in \Omega,
 \end{equation}with both $\bsa$ and $\bsb$ satisfying~\eqref{eq:conditional_2}.
  
 That this is an inner product follows from the conditional positive definiteness of order $m$ of the function $\Phi$.

 To accommodate our polynomial interpolant $s_{f, \bst}$, the space $F_\Phi$ clearly needs to be augmented by  the polynomial space $\Pi_{m-1}$. To facilitate working with polynomials,  following~\cite[Chapter 10]{HW}, letting $Q:= \dim \Pi_{m-1}$ be the dimension of $\Pi_{m-1}$, we now choose a fixed  $\Pi_{m-1}$-unisolvent point set $\bsxi_1, \ldots, \bsxi_Q \in  \Omega$. With respect to this point set let $\{p_1, \ldots,p_Q\}$  be the Lagrange basis for $\Pi_{m-1}$, so that $p_i(\bsxi_j) = \delta_{i,j}$, allowing every polynomial in $\Pi_{m-1}$ to be written uniquely as 
\[
p = \sum_{i = 1}^Q p(\bsxi_i) p_i.
\]
One consequence is that while $\Phi(\bsx, \cdot)$ itself does not belong to $F_\Phi$, the combination $\Phi(\bsx, \cdot) - \sum_{i = 1}^Q p_i(\bsx) \Phi(\bsxi_i, \cdot)$  does belong to $F_\Phi$, because for every polynomial  $p \in \Pi_{m-1}$ we have $p(\bsx) - \sum_{i = 1}^Q p_i(\bsx)p(\bsxi_i) = 0$.

Every element in the sum space $F_\Phi + \Pi_{m-1}$, as noted already,  has a unique representation of the form \eqref{eq:interpolant_rbf}.   An inner product in the space $F_\Phi + \Pi_{m-1}$ is
\begin{align}\label{eq:inner_2}
 \begin{split}
 \langle f,g \rangle_{F_\Phi + \Pi_{m-1}} &=  
 \langle f_\bst + p, g_\bsu +q\rangle_{F_\Phi + \Pi_{m-1}}\\
& := \sum_{k=1}^n \sum_{j = 1}^\ell a_k K_\Phi(\bst_k,\bsu_j) b_j
 +\sum_{i=1}^Q f(\bsxi_i) g(\bsxi_i), 
 \end{split}
 \end {align}where $f_\bst$ and $g_\bsu$ are given by \eqref{eq:ftgu} with $\bsa$ and $\bsb$ satisfying \eqref{eq:conditional_2}, $p, q \in \Pi_{m-1}$, and
\begin{align*}%
 K_\Phi(\bsx, \bsy) &:= \Phi (\bsx, \bsy) - \sum_{i=1}^Q p_i(\bsx) \Phi(\bsxi_i, \bsy) - \sum_{i' = 1}^Q p_{i'}(\bsy) \Phi(\bsx,  \bsxi_{i'})\nonumber\\ 
  & +\sum_{i=1}^Q \sum_{i' = 1}^Q p_i(\bsx)p_{i'}(\bsy)\Phi(\bsxi_i, \bsxi_{i'}).
\end{align*}
Note that with the aid of \eqref{eq:conditional_2} the inner product in $F_\Phi + \Pi_{m-1}$ can also be written as
\begin{align*}%
 \langle f,g \rangle_{F_\Phi + \Pi_{m-1}}
& = \sum_{k=1}^n \sum_{j = 1}^\ell a_k \bigg(\Phi (\bst_k, \bst_j)- \sum_{i=1}^Q p_i(\bst_k) \Phi(\bsxi_i, \bst_j) \\ 
&- \sum_{i' = 1}^Q p_{i'}(\bst_j) \Phi(\bst_k,  \bsxi_{i'})
 +\sum_{i=1}^Q \sum_{i' = 1}^Q p_i(\bst_k) p_{i'}(\bst_j)  \Phi(\bsxi_i, \bsxi_{i'})\bigg)b_j\\ 
&+\sum_{i=1}^Q f(\bsxi_i) g(\bsxi_i)\\
&= \sum_{k=1}^n \sum_{j = 1}^\ell a_k \Phi (\bst_k, \bst_j)b_j +\sum_{i=1}^Q f(\bsxi_i) g(\bsxi_i).
\end{align*}

The linearity of the inner product is obvious, and the positivity with $g = f$ now follows immediately from the conditional positive definiteness of $\Phi$.

We now denote by $\calF_\Phi$ the closure of $F_\Phi$ with respect to the inner product~\eqref{eq:inner_F}, and by $\calN_\phi$ the closure of $F_\phi+\Pi_{m-1}$ with respect to~\eqref{eq:inner_2}, so that 
\begin{equation*}%
\calN_\Phi = \calF_\Phi + \Pi_{m-1}.
\end{equation*}
The space $\calN_\Phi$ is called the ``native space'' associated with the conditionally positive definite basis function $\Phi$.  The native space $\calN_\Phi$ will be our Hilbert space $H$ in the sense of Theorem~\ref{thm:principal}.

It is known (see~\cite[Theorem 10.20]{HW}) that $\mathcal{N}_\Phi$ is a reproducing kernel Hilbert space with reproducing  kernel 
 \begin{equation*}
   K(\bsx, \bsy) := K_\Phi (\bsx, \bsy)  +\sum_{i=1}^Q p_i(\bsx)p_i(\bsy).
 \end{equation*}
 Indeed, it is easy to verify that $K(\bsx, \cdot)$ belongs to the finite dimensional space $F_\Phi + \Pi_{m-1}$, and then that the reproducing property~\eqref{eq:reprod_prop} holds for functions in that space, after which it extends to the native space by taking a limit.
 
 Moreover, it is readily seen that the kernel interpolant $f_n$ for the kernel $K$ satisfies exactly the defining properties \eqref{eq:interpolant_rbf}, \eqref{eq:fa_part}, and \eqref{eq:interp_cond_rbf} assumed for the RBF interpolant $s_{f,\bst}$.  Since the interpolant is unique, it follows that $s_{f,\bst} = f_n$ is the kernel interpolant in the sense of Section~\ref{sec:4}.  
 
 We now define a finite dimensional subspace of $\mathcal{N}_\Phi$ in which the RBF interpolant $s_{f,\bst}$ lies,  
 \begin{equation*}%
 V_\bst:= \left\{ \sum_{k = 1}^n a_k  \Phi(\bst_k, \cdot) +p: p \in \Pi_{m-1},~\bsa \in \bbR^n \mbox{ and satisfying \eqref{eq:conditional_2}},~n  \in \mathbb{N}\right\}.
 \end{equation*}
 From  the reproducing property~\eqref{eq:reprod_prop} we have
\begin{equation*}%
\langle f - s_{f,\bst},  K(\bst_k, \cdot) \rangle_{\mathcal{N}_\Phi} =
f(\bst_k) -  s_{f,\bst}(\bst_k) = 0.
\end{equation*}
Since $\{K(\bst_1,\cdot), \ldots, K(\bst_n,\cdot)\}$ is a basis for $V_\bst$, this proves that $s_{f,\bst}$ is the $\mathcal{N}_\Phi$ orthogonal projection of $f \in \calN_\Phi$ on $V_\bst$.

  Wendland~\cite[Theorems 10.20 and 10.21]{HW} shows that the inner product in $\calN_\Phi$ can be written explicitly  (under conditions satisfied by all the usual basis functions $\Phi$) as 
  \begin{equation}\label{eq:inner_nat}
 \langle f, g\rangle_{\calN_\Phi} := (2\pi)^{-d/2} \int_{\mathbb{R}^d}\frac{\widehat{f}(\omega)\overline{\widehat{g}(\omega)}}{\widehat{\Phi}(\omega)}\,\rd \omega + \sum_{k=1}^Q f(\bsxi_k) g(\bsxi_k), 
 \end{equation}
where $\,\,\widehat{}\,$ denotes  the so-called generalized Fourier transform of order $m$, see~\cite[pg.~103]{HW}.  A crucial  fact, generalizing Bochner's characterisation of positive semidefinite functions as those with non-negative Fourier transforms, is that $\widehat{\Phi}$ is non-negative and non-vanishing for $\Phi$ a slowly increasing continuous and  even, positive definite function of order $m \ge 0$, see~\cite[Theorem 8.12]{HW}.   For simplicity we shall assume that $\widehat{\Phi}$ is a non-increasing radial function.
 
 A useful fact is that the generalized Fourier transform of order $m$ of a polynomial in $\Pi_{m-1}$ vanishes, see~\cite[Proposition 8.10]{HW}, thus $f$ and $g$ in the first term of~\eqref{eq:inner_nat} may without change have arbitrary polynomials in $\Pi_{m-1}$  added. 
 
 There is a subtlety in the definition~\eqref{eq:inner_nat}, namely that the appearance of (generalized) Fourier transform of $f$ implicitly requires that $f$ be defined on the whole of $\mathbb{R}^d$.  Functions in $F_\Phi$ have an obvious extension from $\Omega$ to $\mathbb{R}^d$, and therefore so do functions in $\calN_\Phi$.  The subtle point is that the target function $f$ must also be assumed to have a suitably smooth extension from $\Omega$ to the whole of $\mathbb{R}^d$, which may not be a reasonable assumption in practice.
 
 A second subtlety is that $f(\bsxi_i)$ (like any point evaluation) is not a bounded functional in  $L_2(\Omega)$.   To overcome this technical difficulty we will where necessary below replace the usual $L_2$ norm by the larger norm
  \begin{equation*}%
 \|f\|_{\calL_2(\Omega)}  := \left(\|f\|_{L_2}^2  +\sum_{i = 1}^Q  |f(\bsxi_i)|^2\right)^{1/2}. 
 \end{equation*}

We now take $ H = \calN_\Phi$, which is embedded in $\calL_2(\Omega)$ since
\[ 
\|f\|_{\calL_2(\Omega)}^2 = (2\pi)^{-d/2} \int_{\mathbb{R}^d}|\widehat{f}(\omega)|^2 \rd \omega +\sum_{i = 1}^Q  |f(\bsxi_i)|^2\le \max(\widehat{\Phi}(0),1)\|f\|_{\calN_{\Phi}}^2.
\]
Finally, we come to define $B$: we take $B= \calN_{\Phi*\Phi}$ to be the Hilbert space for which 

 \begin{align}\label{eq:calH_norm}
 \|f\|_{\calN_{\Phi*\Phi}}^2%
 &:=  (2\pi)^{-d/2} \int_{\mathbb{R}^d} \frac{|\widehat{f}(\omega)|^2}{\widehat{\Phi}(\omega)^2} \,\rd \omega\, + \sum_{k=1}^Q |f(\bsxi_k)|^2 < \infty.
 \end{align}We note that the semi-norm in the first term of~\eqref{eq:calH_norm} with its squared denominator seems to have appeared first in Schaback's rate doubling argument~\cite{Sch99}.  We follow here his notation.

That $B = \calN_{\Phi*\Phi}$ is embedded in $H = \calN_\Phi$  follows  in the same way as the embedding of $\calN_\Phi$ in $\calL_2$.

The following proposition will allow us to apply results analogous  to Theorem~\ref{thm:principal} to radial basis function interpolation.

\begin{proposition}\label{prop:norms_rbf}
 Let $\Phi$ be a real, even, continuous and  conditionally positive definite function of order $m\ge 0$ defined on $\mathbb{R}^d$. 
With $H = \calN_\Phi$ and $B = \calN_{\Phi*\Phi}$ we have
\[
|\langle f,g \rangle_H| \le c\|f\|_{\calL_2} \|g\|_B \quad \mbox {for all } f \in H, g \in B,
\]
for some constant $c>0$.
\end{proposition}
\begin{proof}
From~\eqref{eq:inner_nat} and repeated  Cauchy--Schwarz inequalities we obtain
\begin{align*}
|\langle f, g\rangle_H| &:= \bigg|(2\pi)^{-d/2}\int_{\mathbb{R}^d}\frac{\widehat{f}(\omega)\overline{\widehat{g}(\omega)}}{\widehat{\Phi}(\omega)}\,\rd \omega\ + \sum_{i = 1}^Q f(\bsxi_i) g(\bsxi_i)\bigg|\\
&\le \bigg((2\pi)^{-d/2} \int_{\mathbb{R}^d}|\widehat{f}(\omega)|^2\,\rd \omega\bigg)^{1/2}     \bigg((2\pi)^{-d/2}   \int_{\mathbb{R}^d}\frac{|\widehat{g}(\omega)|^2}{{\widehat{\Phi}(\omega)}^2}\,\rd \omega\bigg)^{1/2} \\
&+   \bigg(\sum_{i = 1}^Q |f(\bsxi_i)|^2 \bigg) ^{1/2} \bigg(\sum_{i = 1}^Q |g(\bsxi_i)|^2 \bigg) ^{1/2}
\le c\|f\|_{\calL_2} \|g\|_{B},
\end{align*}
as desired. 
 \end{proof}

Now we can state the following extension  of Theorem~\ref{thm:principal}.

\begin{theorem}\label{thm:rbf}
Let $\Phi$ be a real, even, continuous and  conditionally positive definite function of order $m$ defined on $\mathbb{R}^d$ and let $\Omega$ be a bounded domain in $\mathbb{R}^d$ with Lipschitz boundary.
Let $\bst:= \{\bst_1, \ldots,\bst_n\}$ be a $\Pi_{m-1}$-unisolvent and distinct set of points in $\Omega$. %
Let $H$ be the native space $\calN_\Phi$ with inner product~\eqref{eq:inner_nat}, and let $B$ be the subspace of $H$ with norm~\eqref{eq:calH_norm}.   Finally, let $s_{f,\bst} \in H$ be the radial basis function interpolant of $f$ defined by \eqref{eq:interpolant_rbf}, \eqref{eq:fa_part} and \eqref{eq:conditional_2}. 
 Assume that for some $c>0$ and some $\kappa> 0$ we have
\begin{equation}\label{eq:error_rbf}
 \|f - s_{f,\bst}\|_{L_2} \le c n^{-\kappa} \|f\|_{H} \; \mbox{ for all }    f \in H.  
 \end{equation}
 Then for all $g\in B\subset H$ we have 
 \[
 \|g - s_{g,\bst}\|_{L_2} \le c^2   n^{-2\kappa}\|g\|_{B},  
 \]
 and 
 \[
 \|g - s_{g,\bst}\|_H \le  c n^{-\kappa}\|g\|_B.  
 \]
\end{theorem} 

\begin{proof}
Because $f - s_{f,\bst}$ vanishes on the point set $\bsxi$ (remembering that $\bsxi= (\bsxi_1,\ldots, \bsxi_Q)$ is a subset of the interpolation point set $\bst$), the hypothesis~\eqref{eq:error_rbf}  in the theorem can be replaced by 
\[
 \|f - s_{f,\bst}\|_{\calL_2} \le c n^{-\kappa} \|f\|_{H} \; \mbox{ for all }    f \in H.  
 \]
Thus the conclusions of that theorem hold with the $L_2$ norm replaced by the $\calL_2$ norm.  But since $g - s_{g,\bst}$ also vanishes on the point set $\bsxi$, the results hold with $L_2$ norms, as stated.\end{proof}

\section{Conclusion}\label{sec:conclusion}
We have investigated some consequences of a general condition~\eqref{eq:norms_relation} posed on a Hilbert space $H$ and a ``smoother'' normed space $B$ of real-valued functions. This condition ensures that the $L_2$ rate of convergence for function approximation in $H$ can be doubled for functions belonging to a smoother normed space. We have presented several examples of suitable spaces $H$ and $B$ which fall into this framework, with particular emphasis on applications to kernel interpolation in reproducing kernel Hilbert spaces as well as radial basis function interpolation.

\section*{Acknowledgements}
The authors acknowledge support from the Australian Research Council through grant DP210100831.  They also acknowledge stimulating  discussions with Leszek Demkowicz, Vidar Thom\'ee, Robert Schaback and Holger Wendland, and 
one of us (IHS) acknowledges the support of the Oden Institute in Computational Science and Engineering of the University of Texas, where this work commenced.

 \bibliographystyle{plain}

\end{document}

%% file: common-macros.tex
\usepackage{graphicx}       

\usepackage[latin1]{inputenc}
\usepackage[T1]{fontenc}
\usepackage{microtype} 

\usepackage{cite}
\usepackage{array,colortbl}
\DeclareFontFamily{U}{mathx}{\hyphenchar\font45}
\DeclareFontShape{U}{mathx}{m}{n}{
      <5> <6> <7> <8> <9> <10>
      <10.95> <12> <14.4> <17.28> <20.74> <24.88>
      mathx10
      }{}
\DeclareSymbolFont{mathx}{U}{mathx}{m}{n}
\DeclareFontSubstitution{U}{mathx}{m}{n}
\DeclareMathAccent{\widecheck}{0}{mathx}{"71}

\def\citep#1#2{\cite[{#1}]{#2}}

\usepackage{algorithm}
\usepackage{algorithmic}

\newcommand{\bbR}{{\mathbb{R}}}

\DeclareSymbolFont{bbold}{U}{bbold}{m}{n}
\DeclareSymbolFontAlphabet{\mathbbold}{bbold}

\newcommand{\bsa}{{\boldsymbol{a}}}
\newcommand{\bsb}{{\boldsymbol{b}}}

\newcommand{\bsh}{{\boldsymbol{h}}}

\newcommand{\bst}{{\boldsymbol{t}}}
\newcommand{\bsu}{{\boldsymbol{u}}}

\newcommand{\bsx}{{\boldsymbol{x}}}
\newcommand{\bsy}{{\boldsymbol{y}}}

\newcommand{\bszero}{{\boldsymbol{0}}} 

\newcommand{\bsgamma}{{\boldsymbol{\gamma}}}

\newcommand{\bsxi}{{\boldsymbol{\xi}}}

\newcommand{\calF}{{\mathcal{F}}}

\newcommand{\calL}{{\mathcal{L}}}

\newcommand{\calN}{{\mathcal{N}}}

\newcommand{\setu}{{\mathfrak{u}}}

\newcommand{\rd}{\,\mathrm{d}} 

%% file: Doubling_the_rate_arXiv.bbl
\begin{thebibliography}{99}
 
 \bibitem{Aron} Aronszajn, N.: {Theory of Reproducing Kernels}. \emph{Trans. Amer. Math. Soc.}~\textbf{68}(3), 337--404 (1950)
 
\bibitem{JA67}Aubin, J. P.: 
Behavior of the error of the approximate solutions of boundary value problems for linear elliptic operators by Galerkin's and finite difference methods. \emph{Ann. Sc. Norm. Super. Pisa Cl. Sci.}~\textbf{21}, 599--637 (1967)


\bibitem{ckns21} Cools, R., Kuo, F.~Y., Nuyens, D., Sloan, I.~H.: Fast component-by-component construction of lattice algorithms for multivariate approximation with POD and SPOD weights. \emph{Math. Comp.}~\textbf{90}, 787--812 (2021)


 \bibitem{GRZ19} Griebel, M., Rieger, C., Zastel, P.: Kernel-based stochastic collocation for the random two-phase Navier-Stokes equations. \emph{Int. J. Uncertain. Quantif.}~\textbf {9}, 471--492 (2019)
 
%

 \bibitem{HR22}
 Hangelbroek, T., Rieger, C.: Extending error bounds for radial basis function interpolation to measuring the error in higher order Sobolev norms. To appear in \emph{Math.~Comp.} (2024)%
 
 \bibitem{HLP}
Hardy, G.~H., Littlewood, J.~E., P\'{o}lya, G.: \emph{Inequalities}. Cambridge University Press (1934)
 

\bibitem{KKKNS22} Kaarnioja, V., Kazashi, Y., Kuo, F.~Y., Nobile, F., Sloan, I.~H.: Fast approximation by periodic kernel-based lattice-point interpolation with application in uncertainty quantification. \emph{Numer.~Math.}~\textbf{150}, 33--77 (2022)

 
 \bibitem{KKS23} Kaarnioja, V., Kuo, F.~Y., Sloan, I.~H.: Lattice-based kernel approximation and serendipitous weights for parametric PDEs in very high dimensions. In: A. Hinrichs, P. Kritzer, F. Pillichshammer (eds.), \emph{Monte Carlo and Quasi-Monte Carlo Methods 2022}, pp.~81--103, Springer (2024)
 
\bibitem{MDME24} Mohammadi, M., De Marchi, S., Esfahani, M.~K.: Full-rank orthonormal bases for conditionally positive definite kernel-based spaces. \emph{J. Comput. Appl. Math.}~\textbf{444}, 115761 (2024)

\bibitem{JN68}Nitsche, J.: Ein Kriterium f\"ur die Quasi-Optimalit\"at des Ritzschen Verfahrens. \emph{Numer. Math.}~\textbf{11}, 346--348 (1968)

\bibitem{Sch99}
Schaback, R.: Improved error bounds  for scattered data interpolation by radial basis functions. \emph{Math. Comp.}~\textbf{68}, 201--216 (1999)


\bibitem{GW} Wahba, G.: \emph{Spline Models for Observational Data}. SIAM (1990)

\bibitem{HW} Wendland, H.: \emph{Scattered Data Approximation}. 
Cambridge University Press (2005)

\bibitem{ZKH09} Zeng, X.~Y., Kritzer, P., Hickernell, F.~J.: Spline methods using integration lattices and digital nets. \emph{Constr. Approx.}~\textbf{30}, 529--555 (2009)

\bibitem{ZLH06} Zeng, X.~Y., Leung, K.~T., Hickernell, F.~J.: Error analysis of splines for periodic problems using lattice designs.
In: H. Niederreiter, D. Talay (eds.), \emph{Monte Carlo and Quasi-Monte Carlo Methods 2004}, pp.~501--514, Springer (2006)

\end{thebibliography}
